\documentclass{amsart}
\usepackage{amssymb}
\usepackage{graphics}

\theoremstyle{plain}

\newtheorem{theorem}{Theorem}

\newtheorem*{theorem 1}{Theorem 1}
\newtheorem*{theorem 2}{Theorem 2}
\newtheorem*{theorem 3}{Theorem 3}
\newtheorem*{theorem 4(a)}{Theorem 4(a)}
\newtheorem*{theorem 4(b)}{Theorem 4(b)}
\numberwithin{equation}{section}

\newtheorem{proposition}[equation]{Proposition}

\newtheorem{lemma}[equation]{Lemma}
\newtheorem{corollary}[equation]{Corollary}
\theoremstyle{definition}
\newtheorem{remark}[equation]{Remark}
\newtheorem{example}[equation]{Example}
\newtheorem{definition}[equation]{Definition}

\newcommand {\printname}[1] {}

\def    \R  {{\mathbb R}}
\def    \Z  {{\mathbb Z}}

\def    \C  {{\mathbb C}}

\def     \grad {{\operatorname{grad}}}

\def    \im{{\operatorname{im}}}
\def    \lcm{{\operatorname{lcm}}}

\begin{document}
\title[The fundamental groups of presymplectic Hamiltonian $G$-manifolds]{The fundamental groups of presymplectic Hamiltonian $G$-manifolds}

\author{Hui Li} 
\address{School of mathematical Sciences\\
        Soochow University\\
        Suzhou, 215006, China.}
        \email{hui.li@suda.edu.cn}

\thanks{2010 classification. Primary :  53D05, 53D10, 53D20; Secondary :
55Q05}
\begin{abstract}
We consider presymplectic manifolds equipped with Hamiltonian $G$-actions, $G$ being a connected compact Lie group.
A presymplectic manifold is foliated by the integral submanifolds of 
the kernel of the presymplectic form.  For a presymplectic Hamiltonian $G$-manifold, Lin and Sjamaar propose a condition under which they show that the moment map image has the same ``convex and polyhedral"  property as the moment map image of a symplectic Hamiltonian $G$-manifold, a result proved independently by Atiyah, Guillemin-Sternberg, and Kirwan. In this paper, under the condition Lin and Sjamaar proposed on presymplectic Hamiltonian $G$-manifolds, we study the fundamental groups of such manifolds, comparing with
earlier results on the fundamental groups of symplectic Hamiltonian $G$-manifolds.
 We observe that the results on the symplectic case are special cases of the results on the presymplectic case.
\end{abstract}

 \maketitle

\section{introduction}
Let $M$ be a smooth manifold, $\omega$ be a closed $2$-form on $M$ with constant rank. If $\ker \omega = 0$, then $(M, \omega)$ is a symplectic manifold, otherwise,  $(M, \omega)$ is called a {\bf presymplectic manifold}, and $\omega$ is called a presymplectic form. Symplectic and contact manifolds are special cases of presymplectic manifolds.
Let $G$ be a connected compact Lie group acting on $(M, \omega)$ preserving $\omega$. If  $(M, \omega)$ is symplectic and the $G$-action
is Hamiltonian with a {\it proper} moment map $\phi$, by Atiyah, Guillemin-Sternberg, and Kirwan's theorems,  $\phi(M)\cap \mathfrak t^*_+$
is a closed convex polyhedral set (\cite{A, GS0, K}), where $\mathfrak t^*_+\subset \mathfrak t^*$ is a closed positive Weyl chamber, $\mathfrak t^*$
being the dual Lie algebra of a maximal torus of $G$. 
If $(M, \omega)$ is a presymplectic $G$-manifold, we can similarly define Hamiltonian $G$-actions and moment maps (see Sec. 2). If $(M, \omega)$ is presymplectic and the $G$-action on $M$ is Hamiltonian with a proper moment map $\phi$, then $\phi(M)\cap\mathfrak t^*_+$ may not be convex, or be a polyhedral set (see \cite{LS}). In recent years, Lin and Sjamaar propose a condition, called {\bf cleanness} of the $G$-action, and show that under this condition,  $\phi(M)\cap\mathfrak t^*_+ $ is a closed convex polyhedral set; this set is a convex polytope if $M$ is compact (\cite{LS}). See also \cite{RZ} for a presymplectic convexity result for torus actions by Ratiu and Zung.

In \cite{L00, L0, L}, the author studies the fundamental groups of  symplectic Hamiltonian $G$-manifolds, respectively for circle actions, 
for $G$-actions on compact and noncompact manifolds. The proof relies
in a certain degree on the polyhedral property of the image $\phi(M)\cap\mathfrak t^*_+$ above. In this paper, when the presymplectic Hamiltonian $G$-manifold has this same property, we study the fundamental groups of such manifolds.

Throughout this paper, $G$ always denotes  a connected compact Lie group, $T$ denotes a maximal torus of $G$, $\mathfrak g$ and $\mathfrak t$ are respectively the Lie algebras of $G$ and $T$, $\mathfrak g^*$ and $\mathfrak t^*$ are respectively the dual Lie algebras, and $\mathfrak t^*_+$ denotes a closed positive Weyl chamber in  
$\mathfrak t^*$.

If $(M, \omega)$ is a compact connected symplectic manifold with a Hamiltonian $G$-action and moment map $\phi$, for any $a\in\phi(M)$, we have the symplectic quotient space $M_a$. Due to the existence of different kinds of singular values on  $\phi(M)\cap\mathfrak t^*_+$, the 
$M_a$'s may have many different topological types.
In previous works, we show that the quotient map $M\to M/G$ induces an isomorphism $\pi_1(M)\cong\pi_1(M/G)$, and there are natural isomorphisms $\pi_1(M/G)\cong\pi_1(M_a)$ for all the symplectic quotients $M_a$'s. These results may no longer be true for presymplectic Hamiltonian $G$-manifolds. For example,
$M=S^1$ is a presymplectic manifold with trivial presymplectic form,  let the
circle group $S^1$ act on $M$ freely, then $\pi_1(M)\ncong \pi_1(M/S^1)$.
For another example, by \cite{L1}, the fundamental group of a connected compact contact toric manifold of Reeb type is a finite cyclic group, while the quotient space is contractible (see also Example~\ref{ex-contact}). 
For presymplectic Hamiltonian $G$-manifolds, we can similarly define
presymplectic quotients, see Definition~\ref{quotient};
Examples~\ref{ex1} and \ref{ex2} provide examples for $\pi_1(M/G)\ncong\pi_1(M_a)$ in the presymplectic case.

In this paper, assuming $(M, \omega)$ is a connected presymplectic 
Hamiltonian $G$-manifold, we look at the relations between $\pi_1(M)$,
$\pi_1(M/G)$ and the $\pi_1(M_a)$'s. In Section~\ref{sec2} we will explain the terminologies null subgroup, cleanness of the action, leafwise nontangency  and leafwise transitivity of the action, where the latter two are special cases of the cleanness of the action. Under the condition of the cleanness of the $G$ action, we obtain good results. 
In the following Theorems~\ref{O}, \ref{A} and \ref{B}, we concentrate on compact manifolds. Since $G$ is connected, the quotient map always induces a surjection
$\pi_1(M)\twoheadrightarrow\pi_1(M/G)$. This map may have a nontrivial kernel in the presymplectic case. The kernel comes from the images of
the fundamental groups of the orbits of the null group action --- the group action whose orbits are everywhere tangent to the leaves of the foliation given by 
$\ker(\omega)$.

\begin{theorem}\label{O}
Let $(M, \omega)$ be a connected compact presymplectic Hamiltonian $G$-manifold. If the $G$-action is leafwise nontangent everywhere, then
the quotient map induces an isomorphism $\pi_1(M)\cong\pi_1(M/G)$.
\end{theorem}

In Theorem~\ref{O}, if $(M, \omega)$ is symplectic, then the assumption on the nontangency of the action is automatically satisfied.

\medskip

Suppose $G$ acts on a manifold $M$. Let $M_{(H)}$ be the set of points in $M$ with stabilizer groups conjugate to $H\subset G$, it is called the {\bf $(H)$-orbit type}. 
It is clear that the $G$-orbits in $M_{(H)}$ are diffeomorphic to each other. If $M_{(H)}$ is closed, it is called a {\bf closed orbit type}. 
For more general cases than that in Theorem~\ref{O}, Theorem~\ref{A} gives a description on the kernel of the map $\pi_1(M)\to\pi_1(M/G)$. 

\begin{theorem}\label{A}
Let $(M, \omega)$ be a connected compact presymplectic Hamiltonian $G$-manifold with a clean $G$-action. Let $N$ be the null subgroup. From each closed orbit type of the $N$-action, take an $N$-orbit 
$\mathcal O$. Let $\langle\im\big(\pi_1(\mathcal O)\big)\rangle$ be the normal subgroup of $\pi_1(M)$ generated by the images of the $\pi_1(\mathcal O)$'s in 
$\pi_1(M)$.
Then the quotient map induces an isomorphism $\pi_1(M)/\langle\im\big(\pi_1(\mathcal O)\big)\rangle \cong\pi_1(M/G)$. 
\end{theorem}

Theorem~\ref{A} recovers the case of symplectic manifolds, where $N$ is trivial, see \cite{L0}.

\medskip

 If $(M, \alpha)$ is a contact manifold, where $\alpha$ is a contact $1$-form, if we take $\omega = d\alpha$, then it is 
a presymplectic form. The leaves of the null foliation are the Reeb orbits
of $\alpha$. For a contact $G$-action on $M$, there is automatically a contact moment map $\phi\colon M\to\mathfrak g^*$ given by $\langle \phi, \xi\rangle = \phi^{\xi} = \alpha (\xi_M)$ for any $\xi\in\mathfrak g$, where
$\xi_M$ is the vector field on $M$ generated by $\xi$. This moment map is $G$-equivariant (see e.g.\cite{G}). It is easy to check that up to a sign, this moment map is the same as the one defined in Sec. 2 for the presymplectic $G$-manifold $(M, d\alpha)$.

Let $(M, \alpha)$ be a connected compact contact $G$-manifold. The $G$-action is called of {\bf Reeb type} if there is a Lie algebra element $\xi\in\mathfrak g$ which generates the Reeb vector field of $\alpha$. We can perturb $\alpha$ to a $G$-invariant contact form $\alpha'$ so that 
$\ker (\alpha) = \ker (\alpha')$, and the Reeb vector field of $\alpha'$ is generated by a rational element $\xi'\in\mathfrak g$, which is close to $\xi$, so that $\xi'$ is the generator of the Lie algebra of a circle subgroup of $G$ (see for example \cite{BG}). Clearly, $\pi_1(M)$ and $\pi_1(M/G)$ are independant of the chosen invariant contact form.
For contact $G$-manifolds, we specialize the result as follows. See Example~\ref{ex-contact} for an application.

\begin{theorem}\label{B}
Let $(M, \alpha)$ be a connected compact contact manifold with a clean
$G$-action. Then the action is either trivial or is
of Reeb type. 
When the action is of Reeb type, perturb the contact form to an invariant contact form so that its Reeb orbits are generated by a circle subgroup $S^1\subset G$, and let 
$m=\lcm\big\{k\,|\, \Z_k\subset S^1 \,\mbox{is a stabilizer group}\big\}$, then 
$S^1/\Z_m$ corresponds to a loop $\mathcal O$ in $M$. Let  
$\langle\im\big(\pi_1(\mathcal O)\big)\rangle$ be the normal subgroup of $\pi_1(M)$ generated by
the image of $\pi_1(\mathcal O)$ in $\pi_1(M)$. Then the quotient map induces an isomorphism
$\pi_1(M)/\langle\im\big(\pi_1(\mathcal O)\big)\rangle \cong \pi_1(M/G)$.  
\end{theorem}

\smallskip

For a presymplectic Hamiltonian $G$-manifold, we define the quotient spaces as follows, which is similar to the symplectic quotient spaces.

\begin{definition}\label{quotient}
Let $(M, \omega)$ be a presymplectic Hamiltonian $G$-manifold, and 
let $\phi\colon M\to\mathfrak g^*$ be the $G$-equivariant moment map.
For any $a\in\phi(M)$, define the presymplectic quotient space at $a$ to be $M_a = \phi^{-1}(a)/G_a = \phi^{-1}(G\cdot a)/G$, where $G_a$ is the stabilizer group of $a$ under the coadjoint action.
\end{definition}

Theorem~\ref{conv} in Sec. 2 proved by Lin and Sjamaar states that 
when we have a clean Hamiltonian $G$-action on a connected presymplectic manifold
$M$ with a proper moment map $\phi$, then $\phi(M)\cap\mathfrak t^*_+$ is a closed convex polyhedral set. If this set is not a single point, then there are infinitely many presymplectic quotient spaces; and like in the symplectic case, the existence of possibly different kinds of singular values infers that the quotient spaces may have many different topological types. 
The quotient space $M_a$ in Definition~\ref{quotient} may not be presymplectic or smooth,  similar to the symplectic case. Here we are not looking for a condition to
make a certain $M_a$ to be a smooth presymplectic manifold.
We consider the family $\{M_a\,|\, a\in\phi(M)\}$ as topological spaces, and study their fundamental groups as topological spaces.  

\begin{theorem}\label{C}
Let $(M, \omega)$ be a connected presymplectic Hamiltonian $G$-manifold with a clean $G$-action and proper moment map $\phi$. Then
we have natural isomorphisms $\pi_1(M/G)\cong\pi_1(M_a)$ for all $a\in \phi(M)$.
\end{theorem}
This theorem includes the symplectic case as a special case \cite[Theorem 1.5]{L}.

In Example~\ref{ex0}, we give an application of Theorems~\ref{A} and \ref{C}.

\medskip

The paper is organized as follows. In Section 2, we introduce clean $G$-actions, and the notions appeared in the theorems. We also introduce equivariant moment maps on presymplectic $G$-manifolds. In Section 3,
We prove Theorems~\ref{O} and \ref{A}. In Section 4, we study compact contact $G$-manifolds and prove Theorem~\ref{B}. In Section 5, we consider presymplectic Hamiltonian $G$-manifolds  when the action is leafwise transitive everywhere. In this case, we can use the symplectization of the presymplectic $G$-manifold and previous results on symplectic manifolds to prove Theorem~\ref{C}.  In Section 6, for the proof of Theorem~\ref{C} in the general case, we state the local normal form theorem in the presymplectic case, which is established in \cite{LS}; we address the cross section theorem and a convergence theorem in the presymplectic case. In Section 7, we prove Theorem~\ref{C}. In Section 8,
we give an application of Theorems~\ref{A} and \ref{C}, and we give two counter examples of the theorems for non-clean $G$-actions.

\subsubsection*{Acknowledgement}  
I would like to thank Reyer Sjamaar and Yi Lin for some helpful discussions.
I thank the anonymous referee for a suggestion and some comments.

\section{clean $G$-actions and moment maps on presymplectic $G$-manifolds}\label{sec2}
In this section, we explain the terminologies occurring in the theorems in
the Introduction, and we set up the materials needed for the next sections.

\medskip

Let $(M, \omega)$ be a presymplectic manifold. The distribution 
$$\ker(\omega) = \big\{ v\in T_mM, m\in M\,|\, \omega(v, \cdot)=0\big\}$$
is involutive \cite{BG0}, hence by Frobenius' theorem, it integrates to a regular foliation $\mathcal F$ of $M$,  called the {\bf null foliation}. The leaves of $\mathcal F$ may not be closed.

In  \cite{LS}, Lin and Sjamaar introduce the concept {\it clean $G$-actions}, unifying earlier
considerations of {\it nontangent $G$-actions} by He (\cite{H}) and {\it leafwise transitive torus actions} by Ishida (\cite{I}). Here
we introduce this concept,  one may refer to \cite{LS} for more detail.
Let a connected compact Lie group $G$ act on a presymplectic manifold 
$(M, \omega)$ preserving $\omega$. We call the action {\bf leafwise nontangent everywhere} if
$$T_m(G\cdot m)\cap T_m\mathcal F = 0 \,\,\,\mbox{for each $m\in M$}.$$
Let $\mathfrak n\subset \mathfrak g$ be the Lie subalgebra consisting of all elements $\xi\in\mathfrak g$ such that the induced vector field $\xi_M$ is tangent everywhere to $\mathcal F$, and  let $N$  be the connected immersed Lie subgroup with Lie algebra $\mathfrak n$. It can be shown that the adjoint action of $G$ on $\mathfrak g$ preserves $\mathfrak n$, so 
$\mathfrak n$ is an ideal and hence $N$ is a normal subgroup (\cite{LS}).
The ideal $\mathfrak n$ is called the {\bf null ideal}, and the normal subgroup $N$ is called the {\bf null subgroup}.
We call the $G$-action {\bf leafwise transitive everywhere} if 
$$N\cdot m = \mathcal F(m) \,\,\,\mbox{for each $m\in M$, where $\mathcal F(m)$ is the leaf through $m$}.$$
In \cite{LS}, Lin and Sjamaar defined cleanness of the $G$-action at any point and show that for connected manifold $M$, if the $G$-action is clean at each point on $M$, then the following equation holds: 
\begin{equation}\label{clean}
T_m(G\cdot m)\cap T_m\mathcal F = T_m(N\cdot m) \,\,\,\mbox{for each $m\in M$}.
\end{equation}
In this paper, we only use the notion of cleanness of the action on the whole manifold, so we use this equation as the definition of cleanness of the action on the manifold, i.e., the $G$-action is called {\bf clean on $M$} if (\ref{clean}) holds.

\smallskip
Next, we introduce Hamiltonian $G$-actions on presymplectic manifolds.
The presymplectic $G$-action on $(M, \omega)$ is called {\bf Hamiltonian} if there exists 
a {\bf moment map} 
$\phi\colon M\to \mathfrak g^*$ such that
\begin{itemize}
\item  $i(\xi_M)\omega = d\langle \phi, \xi\rangle$ for each $\xi\in \mathfrak g$, 
where $\xi_M$ is the vector field generated by the $\xi$-action, and
\item  $\phi(g\cdot m)= Ad^*_g(\phi(m))$ for all $g\in G$ and all $m\in M$. This property  is called the equivariance of $\phi$.
\end{itemize}
In the case of Hamiltonian $G$-actions,  the null ideal $\mathfrak n$ satisfies
$$\mathfrak n = \big\{\xi\in \mathfrak g\,|\, \xi_M\in T\mathcal F\big\}=\big\{\xi\in \mathfrak g\,|\, \phi^{\xi} = \langle\phi, \xi\rangle \,\,\mbox{is locally constant on $M$}\big\}.$$
We have the following claim on the image of the moment map for a Hamiltonian $G$-action. Note that this claim does not need the action to be clean.
\begin{proposition}\label{aspan}\cite{LS}
Let $(M, \omega)$ be a connected presymplectic $G$-manifold with
moment map $\phi$. Let $\mathfrak n$ be the null ideal. Then the affine span
of the image $\phi(M)$ is of the form $\lambda + \mathfrak n^{\circ}$,
where $\lambda$ is a fixed point in $\mathfrak g^*$ for the coadjoint action,
and  $\mathfrak n^{\circ}$ is the annihilator of $\mathfrak n$.
\end{proposition}

Under the cleanness assumption of the action, Lin and Sjamaar obtain the following theorem.
\begin{theorem} \cite{LS} \label{conv}
Let $(M, \omega)$ be a connected presymplectic Hamiltonian $G$-manifold
with a clean $G$-action and proper moment map $\phi$. Then the fibers of $\phi$ are connected, $\phi(M)\cap \mathfrak t^*_+$ is a closed convex  polyhedral set, and $\phi(M)\cap\mathfrak t^*_+$ is rational  if and only if the null subgroup $N$ is closed. 
If $M$ is compact, then $\phi(M)\cap\mathfrak t^*_+$  is a convex polytope.
\end{theorem}

Let us clarify some terms occurring in Theorem~\ref{conv}.
A convex polyhedral set in a finite-dimensional real vector space is the intersection of closed half-spaces that is locally finite. A convex polyhedron is the intersection of a finite number of closed half-spaces. A convex polytope is a bounded convex polyhedron.

\section{presymplectic $G$-manifolds and the proof of theorems~\ref{O} and \ref{A}}

In this section, we prove Theorems~\ref{O} and \ref{A}.

\smallskip

Let $(M, \omega)$ be a connected presymplectic Hamiltonian $G$-manifold with moment map $\phi$. Let $g$ be a $G$-invariant  Riemannian metric on $M$ compatible with $\omega$, which means that $g$ is an invariant metric such that on the symplectic subbundle $(T\mathcal F)^{\perp}$ perpendicular to $T\mathcal F$, $g$
is compatible with $\omega|_{(T\mathcal F)^{\perp}}$, i.e., for any
$X, Y\in (T\mathcal F)^{\perp}$, $g(X, Y) = \omega (JX, Y)$ for a $G$-invariant  almost complex structure $J$ on $(T\mathcal F)^{\perp}$  determined by $g$.  Let $\xi\in\mathfrak g$, and let $\phi^{\xi} = \langle\phi, \xi\rangle$ be a moment map component. Let $\grad (\phi^{\xi})$ be
the gradient vector field of $\phi^{\xi}$, i.e., for any $X\in TM$, 
$g(\grad (\phi^{\xi}), X) = d\phi^{\xi} (X)$. 
It is easy to check that
$$\grad (\phi^{\xi}) = - J \bar \xi_M,$$
 where $\bar \xi_M$ is the projected 
vector field of $\xi_M$ to $(T\mathcal F)^{\perp}$.

\medskip

Now we prove Theorem~\ref{O} in the Introduction.
\begin{theorem 1}
Let $(M, \omega)$ be a connected compact presymplectic Hamiltonian $G$-manifold. If the $G$-action is leafwise nontangent everywhere, then
the quotient map induces an isomorphism $\pi_1(M)\cong\pi_1(M/G)$.
\end{theorem 1}

\begin{proof}
First, since $G$ is connected, by the path lifting property (see \cite[Theorem 6.2 or Corollary 6.3]{B}), the quotient map induces a surjection
$$\pi_1(M)\twoheadrightarrow\pi_1(M/G).$$

To show that the quotient map induces an injection, we need to show that
if a loop in $M$ projects to a homotopically trivial loop in $M/G$, then the loop in $M$ is homotopically trivial. By the homotopy lifting property \cite[Theorem 7.3]{B}, we only need to show that the lifting of any constant loop (i.e., a point) in $M/G$, which is a loop in a $G$-orbit in $M$, is homotopically trivial in $M$. Let $G\cdot x$ be any $G$-orbit and $\alpha$ be a loop in this orbit.
Identify $G\cdot x$ with $G/H$, where $H$ is the stabilizer group of $x$. 
Then $\alpha$ lifts to a path in $G$ from
the identity to some $h\in H$. Let $T$ be the maximal torus of $G$ containing the identity and $h$. We can deform the path in $G$, fixing the endpoints, to a path
in $T$, and then deform the path in $T$ to the form $\exp (t\xi)$, where
$t\in [0, 1]$, $\xi\in\mathfrak t$, and $\exp\xi = h$. Then $\alpha$ is homotopic to the loop $\exp (t\xi)\cdot x$.  If this loop is not a constant, then along the flow of $\grad(\phi^{\xi})$ or
$-\grad(\phi^{\xi})$ with respect to some $G$-invariant metric, it can be deformed to a loop
$\exp (t\xi)\cdot x_0$ in the minimum set of $\phi^{\xi}$. Since for any point $m\in \exp (t\xi)\cdot x_0$, $d\phi^{\xi}(m) = 0$, so $T_m(\exp (t\xi)\cdot x_0)\in T_m\mathcal F$. Since the action is nontangent everywhere, $T_m(\exp (t\xi)\cdot x_0)=0$, hence $\exp (t\xi)\cdot x_0$ is a constant loop.
\end{proof}

\begin{remark}
Here are easy examples of leafwise nontangent actions.
Let $(M_1, \omega_1)$ be a connected compact symplectic manifold with a nontrivial symplectic (or Hamiltonian) action of $G$, and let $M_2$ be any connected compact smooth manifold with a trivial $G$-action. Let $M=M_1\times M_2$ with $\omega=\pi^*\omega_1$, where $\pi\colon M\to M_1$ is the projection. Then $(M, \omega)$ is a 
presymplectic manifold with a {\it nontrivial} action of $G$ which is leafwise nontangent everywhere. (Compare these examples with Proposition~\ref{nontang}.)
\end{remark}

Recall that if $M$ is a $G$-manifold, $M_{(H)}$ denotes the $(H)$-orbit type, and  all the $G$-orbits in $M_{(H)}$ are diffeomorphic to each other.
If $M_{(K)}$ is in the closure of $M_{(H)}$, then $(H)\subset (K)$, i.e., 
$H$ is conjugate to a subgroup of $K$. 
Let $O\subset M_{(K)}$ be a $G$-orbit, and suppose that a tubular neighborhood of $O$ intersects $M_{(H)}$. 
Then by the slice theorem, we can see that a tubular neighborhood $N_O$ of $O$ deformation retracts to $O$, so a loop  $\alpha'$ in a $G$-orbit $O'$ in $M_{(H)}\cap N_O$ can be deformed to a loop $\alpha$ in the orbit $O$, so that 
$[\alpha']=[\alpha]\in\pi_1(M)$.

\medskip

Now we are ready to prove Theorem~\ref{A}. Let us recall it first.

\begin{theorem 2}
Let $(M, \omega)$ be a connected compact presymplectic Hamiltonian $G$-manifold with a clean $G$-action. Let $N$ be the null subgroup. From each closed orbit type of the $N$-action, take an $N$-orbit 
$\mathcal O$. Let $\langle\im\big(\pi_1(\mathcal O)\big)\rangle$ be the normal subgroup of $\pi_1(M)$ generated by the images of the $\pi_1(\mathcal O)$'s in 
$\pi_1(M)$.
Then the quotient map induces an isomorphism $\pi_1(M)/\langle\im\big(\pi_1(\mathcal O)\big)\rangle \cong\pi_1(M/G)$. 
\end{theorem 2}

\begin{proof}
As before, since $G$ is connected, the map 
$$q_*\colon \pi_1(M)\to\pi_1(M/G)$$
 induced by the quotient map is surjective. Since 
for each $\mathcal O$, $\im\left(\pi_1(\mathcal O)\right)$ maps to the trivial
element by $q_*$, $\langle\im\left(\pi_1(\mathcal O)\right)\rangle$ maps to the trivial element by $q_*$,  so the map induced by the quotient
$$\pi_1(M)/\langle\im\left(\pi_1(\mathcal O)\right)\rangle\to\pi_1(M/G)$$ is well defined and surjective. This shows that $\langle\im\left(\pi_1(\mathcal O)\right)\rangle\subset \ker (q_*)$. To show 
$\ker (q_*)\subset \langle\im\left(\pi_1(\mathcal O)\right)\rangle$, by the homotopy lifting property \cite[Theorem 7.3]{B}, we only need to show that any loop in any $G$-orbit is homotopic
to a loop in some $\mathcal O$. Let $\alpha$ be a loop at $x$
in the $G$-orbit $G\cdot x$. 
As in the proof of Theorem~\ref{O}, we can deform $\alpha$ to a loop at $x$ of the form $\exp(t\xi)\cdot x$, where $t\in [0, 1]$, and $\xi\in\mathfrak t$; moreover, if $\exp(t\xi)\cdot x$ is not a constant, it can be further deformed to a loop  $\exp(t\xi)\cdot x_0$, either a constant or in the minimum set of $\phi^{\xi}$. Since for each point
$m\in\exp(t\xi)\cdot x_0$, $d\phi^{\xi}(m)=0$, we have $\xi_{M, m} \in T_m(N\cdot m) \subset T_m\mathcal F$ by cleanness of the action. So the loop $\exp(t\xi)\cdot x_0$ is in an 
$N$-orbit (which is in some orbit type of the $N$-action). By the paragraph above
the statement of Theorem~\ref{A},
the loop  $\exp(t\xi)\cdot x_0$ can be further deformed to a loop in an $N$-orbit $\mathcal O$ lying in a closed orbit type of the $N$-action.
\end{proof}

For the special case of compact symplectic manifolds, these arguments shortens those in \cite{L0}. 

For the case of noncompact connected presymplectic Hamiltonian $G$-manifolds, in certain cases, one can modify the statements from the
compact case.

\section{contact $G$-manifolds and the proof of Theorem~\ref{B}}

In this section, we study contact $G$-manifolds and prove Theorem~\ref{B}.

\begin{lemma}\label{nullact}
Let $(M, \alpha)$ be a connected contact $G$-manifold. 
Then for each $\xi\in\mathfrak n$, either $\xi_{M, m}=0$ for all $m\in M$ or $\xi_{M, m}\neq 0$ for all $m\in M$. Here, 
$\mathfrak n$ is the null ideal, and $\xi_M$ is the vector field on $M$ generated by $\xi$.
\end{lemma}

\begin{proof}
If $\mathfrak n =0$, the claim is trivial.
Now assume $\mathfrak n \neq 0$, and let $\xi\in \mathfrak n$. Then by the definition of $\mathfrak n$,  $\xi_M$ is tangent
to the Reeb orbit  (the leaf of $\ker(\omega=d\alpha)$) at any point, so 
$\phi^{\xi}$ is a constant on $M$ ($M$ is connected). While for any point $m\in M$, 
$\phi^{\xi}(m)=\alpha_m (\xi_{M, m})$, so either $\xi_{M, m}=0$ or $\xi_{M, m}\neq 0$ for all $m$.
\end{proof}

The next result says that on compact contact $G$-manifolds, 
if the action is leafwise nontangent everywhere, then the action is trivial.

\begin{proposition}\label{nontang}
Let $(M, \alpha)$ be a connected compact contact $G$-manifold.
If the action is leafwise nontangent everywhere, then the action is trivial.
\end{proposition}

\begin{proof}
Let $\phi$ be the moment map. By Theorem \ref{conv}, $\phi(M)\cap\mathfrak t^*_+$ is a convex polytope. Choose a $G$-invariant metric on 
$\mathfrak g^*$. Let $v$ be a vertex on the polytope furthest from the origin. Let $\phi_T$ be the moment map for the $T$-action. 
Then the image of $\phi_T$ is the orthogonal projection to $\mathfrak t^*$ of the image of $\phi$.
Then $v$ is an extremal value of $\phi_T$ and $\phi^{-1}(v)=\phi_T^{-1}(v)$.  For each $\xi\in\mathfrak t$, and each 
$m\in\phi_T^{-1}(v)$, $d\phi^{\xi}(m) =0$, so $\xi_M(m)\in T_m\mathcal F$, which implies 
$\xi_M(m) = 0$ since the action is nontangent everywhere. 
Then $\phi^{\xi} (m) = \alpha (\xi_M)(m)=0$ for all $\xi\in\mathfrak t$ and all $m\in\phi_T^{-1}(v)$, which means that $\phi_T \big(\phi_T^{-1}(v)\big) = v = 0$. By the choice of $v$, we have $\phi(M)=0$. Since each point is fixed by
the $T$-action, each point is fixed by the $G$-action.
\end{proof}

\begin{proposition}\label{p1}
Let $(M, \alpha)$ be a connected compact contact $G$-manifold with a clean $G$-action. Then the action is either trivial or is of Reeb type.
\end{proposition}

\begin{proof}
Since the action is clean, $T_m(G\cdot m)\cap T_m\mathcal F = T_m(N\cdot m)$ at each $m\in M$. By Lemma~\ref{nullact},
the action  is either leafwise nontangent everywhere or is of Reeb type (leafwise transitive everywhere). By Proposition~\ref{nontang}, the action is trivial when it is leafwise nontangent everywhere. 
\end{proof}

\begin{remark}\label{r1}
Note that when a $G$-action on a contact manifold is of Reeb type, the null subgroup (which is non trivial) is a normal subgroup of $G$.
For example, $SU(2)$ (or $SO(3)$) has no nontrivial connected normal subgroups, 
so $SU(2)$ (or $SO(3)$) cannot act on a  contact manifold which is of Reeb type. By Proposition~\ref{p1}, there are no nontrivial clean $SU(2)$ (or $SO(3)$) actions on compact contact manifolds.
\end{remark}

We next recall Theorem~\ref{B} and give its proof.

\begin{theorem 3}
Let $(M, \alpha)$ be a connected compact contact manifold with a clean
$G$-action. Then the action is either trivial or is
of Reeb type. 
When the action is of Reeb type, perturb the contact form to an invariant contact form so that its Reeb orbits are generated by a circle subgroup $S^1\subset G$, and let 
$m=\lcm\big\{k\,|\, \Z_k\subset S^1 \,\mbox{is a stabilizer group}\big\}$, then 
$S^1/\Z_m$ corresponds to a loop $\mathcal O$ in $M$. Let  
$\langle\im\big(\pi_1(\mathcal O)\big)\rangle$ be the normal subgroup of $\pi_1(M)$ generated by
the image of $\pi_1(\mathcal O)$ in $\pi_1(M)$. Then the quotient map induces an isomorphism
$\pi_1(M)/\langle\im\big(\pi_1(\mathcal O)\big)\rangle \cong \pi_1(M/G)$.  
\end{theorem 3}

\begin{proof}
By Proposition~\ref{p1},  the action is either trivial or is of Reeb type.
Now assume the action is of Reeb type, and the contact form is perturbed as stated in the theorem, we will show the rest of the claims.

As before, since $G$ is connected, the quotient map $q\colon M\to M/G$
induces a surjection
$$q_*\colon \pi_1(M)\twoheadrightarrow \pi_1(M/G).$$
If the null subgroup $S^1$ acts freely, then $m=1$, and  $S^1$ corresponds to a Reeb orbit which is a loop $\mathcal O$. Clearly, $\im(\pi_1(\mathcal O))\subset\ker(q_*)$. Since $\ker(q_*)$ is a normal subgroup, $\langle\im(\pi_1(\mathcal O))\rangle\subset\ker(q_*)$.
Now assume $S^1$ acts locally freely. If $\Z_k$ is a stabilizer group for the $S^1$-action, then $S^1/\Z_k$ represents a Reeb orbit. Since $M$ is compact, we have finitely many distinct stabilizer groups $\Z_{k_i}$s' for the Reeb orbits, 
where $i=0, 1, \cdots, l$. Then $\frac{m}{k_i}$ are relatively prime, so
there are integers $a_0, \cdots, a_l$, such that
$$\sum_{i=0}^l a_i \frac{1}{k_i} = \frac{1}{m}.$$
For $i=0, \cdots, l$,
let $\mathcal O_i\approx S^1/\Z_{k_i}$ be a Reeb orbit,  and let $x_i\in\mathcal O_i$ be a point. Let $\beta_i$ be a path from $x_0$ to $x_i$. Let   
$$\mathcal O =\big(\mathcal O_0^{a_0}\big)\cdot\big(\beta_1\cdot \mathcal O_1^{a_1}\cdot\beta_1^{-1}\big)\cdot \big(\beta_2\cdot \mathcal O_2^{a_2}\cdot\beta_2^{-1}\big)\cdot\cdots\cdot\big(\beta_l\cdot\mathcal O_l^{a_l}\cdot\beta_l^{-1}\big).$$
Then $\mathcal O$ is a loop based at $x_0$, it is parametrized by $S^1/\Z_m$.
Clearly, $q_*\big((\pi_1(\mathcal O_i)\big) = 0$. So $q_*\big((\pi_1(\mathcal O)\big) = 0$, hence $\im\big(\pi_1(\mathcal O)\big)\subset \ker (q_*)$. Hence
$\langle\im\big(\pi_1(\mathcal O)\big)\rangle\subset \ker (q_*)$ since $\ker (q_*)$
is normal. Therefore, the quotient map induces a surjective homomorphism
$$\pi_1(M)/\langle\im\big(\pi_1(\mathcal O)\big)\rangle\twoheadrightarrow\pi_1(M/G).$$ 
To show that  $\ker (q_*)\subset \langle\im\big(\pi_1(\mathcal O)\big)\rangle$, by the homotopy lifting property  \cite[Theorem 7.3]{B}, we need to show that any loop in any $G$-orbit is homotopic to a loop in $\mathcal O$. By Theorem~\ref{A}, the loop in the $G$-orbit is homotopic to a loop in some $\mathcal O_i$. We claim that the generator of $\pi_1(\mathcal O_i)$ is a multiple of the generator of $\pi_1(\mathcal O)$, hence $\ker (q_*)\subset \langle\im\big(\pi_1(\mathcal O)\big)\rangle$ follows.
The claim can be shown in the following way. We consider the map
$$(q_1)_*\colon \pi_1(M)\to \pi_1(M/\Z_m)$$
induced by the quotient $q_1\colon M\to M/\Z_m$ (not a covering!).
This map is injective. To see this, by the homotopy lifting property, we only need to show that the lift of any constant loop in $M/\Z_m$ is homotopically trivial in $M$. This is clear since the lift is a point in a $\Z_m$-orbit. The group $S^1/\Z_m$
acts freely on $M/\Z_m$, any Reeb orbit in $M/\Z_m$ is parametrized by 
$S^1/\Z_m$. It is clear that  $(q_1)_*([\mathcal O])$ is a Reeb orbit in $M/\Z_m$.
Since $(q_1)_*([\mathcal O_i])$ is a multiple of $(q_1)_*([\mathcal O])$, $[\mathcal O_i]$ is the same multiple of  $[\mathcal O]$.
\end{proof}

Recall that a {\bf contact  toric} manifold is a contact manifold of dimension
$2n+1$ with an effective $T^{n+1}$-action. Contact toric manifolds are the odd dimensional analog of symplectic toric manifolds.

\begin{example}\label{ex-contact}
In \cite{L1}, we have shown that a connected compact contact toric manifold $M$ of Reeb type has finite cyclic fundamental group $\Z_k$ for some $k\geq 1$.
Suppose that the contact form is chosen so that its Reeb orbit is generated by
an $S^1$-action. For simplicity, we assume that $S^1$ acts freely. Then we have an
$S^1$ principal bundle $M\to M/S^1$, and the long exact sequence
$$\cdots\to \pi_2(M/S^1)\overset{\partial}{\to} \pi_1(S^1)\to \pi_1(M)\to\pi_1(M/S^1)\to \pi_0(S^1)=0\to\cdots.$$
We see that $\pi_1(M)/\im\big(\pi_1(S^1)\big) = \pi_1(M/S^1)$. Since $M/S^1$
is a compact connected symplectic toric manifold, $\pi_1(M/S^1)=0$. Since $M/T$ is contractible, 
$\pi_1(M/T) = 0$. So we have 
$$\pi_1(M)/\im\big(\pi_1(S^1)\big) = \pi_1(M/T) = 0.$$
Here, $\im\big(\pi_1(S^1)\big)=\im\big(\pi_1(\mathcal O)\big)$, where $\mathcal O$ is a Reeb orbit, parametrized by $S^1$.
Since $\pi_1(M) = \Z_k$, we have 
$$\im\big(\pi_1(S^1)\big) = \im\big(\pi_1(\mathcal O)\big) = \Z_k.$$ 

For the above case  that the Reeb orbits are free $S^1$-orbits, we can also see in the following way  that $\pi_1(M)$ is finite cyclic.
By the long exact sequence above, we have 
$$\pi_1(M) = \pi_1(S^1)/\im\big(\partial \big(\pi_2(M/S^1)\big)\big),$$ 
where $\pi_2(M/S^1) \cong H_2(M/S^1; \Z)$ since $\pi_1(M/S^1)=0$. Since $M/S^1$ is a compact symplectic manifold, $H_2(M/S^1; \Z)$ has dimension at least $1$; $H_2(M/S^1; \Z)$ has no torsion since $M/S^1$ is toric. If the Euler class 
of the $S^1$ bundle $M\to M/S^1$ is $e = a_1 t_1 + \cdots + a_l t_l$, where
$a_i\in \Z$ for all $i$, and $\{t_1, \cdots t_l\}$ is a basis of $H_2(M/S^1; \Z)$, then
$\im\big(\partial \big(\pi_2(M/S^1)\big)\big) = k\Z$, where $k=\gcd(a_1, \cdots, a_l)$.
So $\pi_1(M) = \Z_k$.
\end{example}

\section{proof of Theorem~\ref{C} when the $G$-action is leafwise transitive}

We observe that when the $G$-action is leafwise transitive everywhere, we can use the symplectization of the presymplectic $G$-manifold and previous results on 
symplectic manifolds to prove Theorem~\ref{C}. The reason is that when the Hamiltonian $G$-action is leafwise transitive everywhere on the presymplectic manifold with a proper moment map, the moment map on its symplectization is  proper onto its image.
In general case, however, I do not know how to do it  this way.

 We now introduce the symplectization of a presymplectic $G$-manifold 
$(M, \omega)$.
Let $T^*\mathcal F$ be the dual bundle of $T\mathcal F$ and let $\pi\colon 
T^*\mathcal F\to M$ be the projection. We choose a $G$-invariant Riemannian metric on $M$. Let $TM\to T\mathcal F$ be the orthogonal 
projection, and let $j\colon T^*\mathcal F\to T^*M$ be the dual embedding. 
Let $\omega_0$ be the standard symplectic form on $T^*M$, and let
$$\Omega = \pi^*\omega + j^*\omega_0.$$ 
Then $\Omega$ is symplectic 
in a neighborhood of $M$ in $T^*\mathcal F$. The $G$-action lifts to
$T^*M$ and to $T^*\mathcal F$. The germ at $M$ of the $G$-manifold $T^*\mathcal F$ is called the {\bf symplectization of $M$}, denoted {\bf $X$}. 
The lifted $G$-action on $T^*M$ is Hamiltonian with moment map 
$\phi_0\colon T^*M\to\mathfrak g^*$
given by
\begin{equation}\label{phi0} 
\langle\phi_0(m, y), \xi\rangle = \langle y, \xi_M(m)\rangle,
\end{equation} 
where $m\in M$, $y\in T^*_m M$, $\xi\in \mathfrak g$, and $\xi_M$ is the vector field on $M$ generated by $\xi$.
Moreover, if the $G$-action on $M$ is Hamiltonian with moment map $\phi$, then the lifted $G$-action on $X$ is Hamiltonian, with moment map
\begin{equation}\label{momentS}
\psi = \pi^*\phi + j^*\phi_0.
\end{equation}

\begin{lemma}\label{X}
Let $M$ be a connected presymplectic manifold with a clean Hamiltonian $G$-action and moment map $\phi$. Let $X$ be the symplectization of $M$, and $\psi$ be the moment map for the lifted $G$-action on $X$. Then for any $(m, y)\in X$ with 
$m\in M$ and $y\in T_m^*\mathcal F$, we have 
$$\psi(m, y) - \phi(m) = \Lambda_m (y),$$ 
where $\Lambda_m\colon T^*_m\mathcal F \to \mathfrak g^*$ 
is a linear map such that
\begin{itemize}
\item $\ker(\Lambda_m) = \{ y\in T^*_m\mathcal F\,|\, y|_{T_m(N\cdot m)} = 0\}$, and
\item $\im(\Lambda_m) \cap \mathfrak n^{\circ} = 0$.
\end{itemize}
Here, $N$ is the null subgroup and $\mathfrak n$ is the null ideal.
\end{lemma}

\begin{proof}
By (\ref{momentS}) and (\ref{phi0}), $\psi^{\xi}(m, y) - \phi^{\xi}(m) = \langle y, \xi_M (m)\rangle$ for any $\xi\in\mathfrak g$. We write 
$$\langle y, \xi_M (m)\rangle = \Lambda_m^{\xi}(y)= \langle\Lambda_m (y), \xi   \rangle,$$
then $\Lambda_m (y)$ is linear on $y$.
Then $y\in T^*_m\mathcal F$ is in $\ker(\Lambda_m)$ if and only if
$\langle y, \xi_M (m)\rangle=0$ for all $\xi\in\mathfrak g$, and if and only if $y\in \big(T_m(N\cdot m)\big)^{\circ}$ since
the action is clean.
Now suppose $\Lambda_m (y)\in \im(\Lambda_m) \cap \mathfrak n^{\circ}$. Then $\Lambda_m^{\xi} (y) = \langle y, \xi_M (m)\rangle = 0$ for all
$\xi\in \mathfrak n$. Since $y\in T^*_m\mathcal F$, $\langle y, \xi_M (m)\rangle = 0$ for all $\xi\in\mathfrak g$, hence $\Lambda_m (y) = 0$.
\end{proof}

\begin{lemma}\label{symp}
Let $(M,\omega)$ be a connected presymplectic Hamiltonian $G$-manifold
with proper moment map $\phi$. Suppose that the $G$-action is leafwise transitive
everywhere. Then the moment map of the 
$G$ action on the symplectization $X$ of $M$ is proper onto its image.
\end{lemma}

\begin{proof}
By Lemma~\ref{X} and Proposition~\ref{aspan}, the moment map of the $G$-action on $X$ is 
$$\psi (m, y) = \phi (m) + \Lambda_m(y)\colon X\to \mathfrak n^{\circ}\oplus \mathfrak n^*.$$
By assumption, $\phi$ is proper. The fact that the action is leafwise transitive implies that $\ker(\Lambda_m) = 0$, so $\Lambda_m$ is proper onto its image.
So $\psi\colon X\to \mathfrak g^*$ is proper onto its image.
\end{proof}

A special case for Theorem~\ref{C} is as follows.
\begin{theorem}\label{C'}
Let $(M,\omega)$ be a connected presymplectic Hamiltonian $G$-manifold
with proper moment map $\phi$. Suppose that the $G$-action is leafwise transitive
everywhere. Then $\pi_1(M/G) = \pi_1(M_a)$ for all $a\in\phi(M)$.
\end{theorem}

\begin{proof}
By Lemma~\ref{symp}, the moment map of the $G$-action on $X$ is proper onto its image. By the symplectic case \cite[Theorem 1.5 and Remark 1.7]{L},
$$\pi_1(X/G) = \pi_1(X_a)\quad\mbox{for any $a\in\psi(X)$}.$$
It is clear that $X/G$  is homotopy equivalent to $M/G$. Hence
$$ \pi_1(X/G) = \pi_1(M/G).$$
 If $a\in\phi(M)$, then $a\in\psi(X)$, and $\psi^{-1}(a)=\phi^{-1}(a)$, so
 $X_a = M_a$. Therefore 
$$\pi_1(X_a) = \pi_1(M_a)\quad\mbox{for any $a\in\phi(M)$}.$$
The claim follows.
\end{proof}

\begin{remark}
In the proof above, the symplectic manifold $X$ is not compact, we may not have $\pi_1(X) = \pi_1(X/G)$.
\end{remark}

\section{the local normal form theorem, the cross section theorem, and a convergence theorem}

For the purpose of proving Theorem~\ref{C}, in this section, we address three important theorems for presymplectic Hamiltonian $G$-actions.

\medskip

First let us describe the local normal form theorem for a presymplectic Hamiltonian $G$-manifold $(M, \omega)$ with a clean $G$-action and moment map $\phi$.
 Let $x\in M$, $H$ be the stabilizer group of $x$, and $G_{\phi(x)}$ be the stabilizer group of $\phi(x)$ under the coadjoint action. By the equivariance of $\phi$, $H\subseteq G_{\phi(x)}$. Let 
$\mathfrak h = $ Lie$(H)$, and $\mathfrak g_{\phi(x)}=$ Lie$(G_{\phi(x)})$ be the Lie algebras. Let  $\mathfrak m = \mathfrak g_{\phi(x)}/\mathfrak h$. Let  $\mathfrak p = \mathfrak n/(\mathfrak n\cap\mathfrak h)$,  where $\mathfrak n$ is the null ideal. By Proposition~\ref{aspan}, $\mathfrak n\subset \mathfrak g_{\phi(x)}$, so we may take  $\mathfrak q = \mathfrak m/\mathfrak p$, let $\mathfrak q^*$ be the dual of $\mathfrak q$. Using these notations, we can describe 
the local normal form theorem as follows, the theorem is established in \cite[Appendix C]{LS}.
\begin{theorem} \label{nform} (The local normal form)
Let $(M, \omega)$ be a presymplectic Hamiltonian $G$-manifold with a clean $G$-action and moment map $\phi$.  Let $x\in M$, with stabilizer group $H$.
Then a $G$-invariant neighborhood
of the orbit $G\cdot x$ in $M$ is isomorphic as a presymplectic Hamiltonian $G$-manifold to 
$$A = G\times_H (\mathfrak q^*\times S\times V),$$ 
where $S$ is the ``symplectic slice" on which $H$ acts symplectically, $V$ is the ``null slice", which is a linear $H$-invariant subspace of $T_x \mathcal F$. The moment map on $A$ is
$\phi_A ([g, a, s, v]) = Ad(g)^*\big(\phi(x)+a+\psi(s)\big)$, where $\psi$ is the moment map for the $H$-action on $S$.
\end{theorem}

Note that the ``null slice" $V$ has no contribution to the moment map image.

In the local normal form theorem for a symplectic Hamiltonian $G$-manifold,  there is no ``null slice" $V$,  and in place of  $\mathfrak q^*$ above, it is $\mathfrak m^*$, the dual space of the above $\mathfrak m$, see
\cite{GS, M}.

\medskip

Next, similar to the symplectic Hamiltonian $G$-action case, we establish
the cross section theorem for presymplectic Hamiltonian $G$-actions, where $G$ is nonabelian.

Suppose that a Lie group $G$ acts on a manifold $M$. Given a point $m$ in $M$ with stabilizer group $G_m$,
   a submanifold $U\subset M$ containing $m$ is called a
   {\bf slice at m} if $U$ is $G_m$-invariant, $G\cdot U$ is a neighborhood of $m$, and the map
 $$G\times_{G_m}U \longrightarrow G\cdot U,\,\,\, \mbox{with $[g, u]\longrightarrow g\cdot u$},$$
is an isomorphism.  For instance,  consider the coadjoint action of $G$ on $\mathfrak g^*$. Let  $a\in \mathfrak t^*_+$. Let $\tau$ be the open face of  
$\mathfrak t^*_+ $ containing
 $a$  and let $G_a$ be the stabilizer group of $a$. Since all the points on $\tau$ have the same stabilizer group,
 we also use $G_{\tau}$ to denote $G_a$. Then the natural  slice at $a$ is
$U = G_a\cdot\big\{b\in \mathfrak t^*_+ \,|\,G_b\subset
G_a\big\}=G_a\cdot\bigcup_{\tau\subset\overline{\tau'}}\tau'$,
   and it is an open subset of $\mathfrak g_{\tau}^*=\mathfrak g_a^*$.\\

We have the following cross section theorem. The cross section theorem 
in the symplectic case is due to Guillemin and Sternberg \cite[Theorem 26.7]{GS1},
also see \cite[Cor. 2.3.6]{GLS} and \cite[Theorem 3.8]{LMTW}.

\begin{theorem}\label{cross}(The cross section)
 Let $(M, \omega)$ be a connected presymplectic Hamiltonian $G$-manifold with moment map $\phi$. Let $a\in\phi(M)\cap \mathfrak t^*_+$, let $U$ be the natural slice at $a$, and $G_a$ be the stabilizer group of $a$. Then the {\bf cross section} $R=\phi^{-1}(U)$ is a $G_a$-invariant presymplectic submanifold of $M$, and every leaf of $R$ is a leaf of $M$; the restriction $\phi|_R$ is a moment map for the action of $G_a$ on $R$. If the $G$-action is clean on $M$, then the $G_a$-action is clean on $R$.
\end{theorem}

\begin{proof}
The proof of the claims on the cross section is similar as in the symplectic case, we refer to \cite[Theorem 3.8]{LMTW} if more detail is necessary. Here we outline the main points of the proof.  
First, since the coadjoint orbits through $U$ intersect $U$ transversely, and $\phi$
is equivariant, $\phi$ is transverse to $U$, so $R=\phi^{-1}(U)$ is a submanifold. Since $U$ is $G_a$-invariant and $\phi$ is equivariant, $R$ is $G_a$-invariant.
Let $\mathfrak f$ be the orthogonal complement of $\mathfrak g_a$ in 
$\mathfrak g$ with respect to some $G$-invariant metric, so
$\mathfrak g = \mathfrak g_a \oplus \mathfrak f $. Let 
$$\mathfrak f_{M, x} = \big\{\xi_{M, x}\,|\, \xi\in\mathfrak f, x\in R\big\}.$$ 
For $x\in R$, we will check the following two things:
\begin{enumerate}
\item $T_x R$  is symplectically perpendicular to 
$\mathfrak f_{M, x}$ in $T_xM$, and 
\item $\mathfrak f_{M, x}$ is a symplectic subspace of $T_x M$. 
\end{enumerate}
For $(1)$,
let $\xi\in\mathfrak f$, and $v\in T_x R$, then
$$\omega(\xi_{M, x}, v) = d\phi (v) (\xi) = 0$$
since $d\phi(v)\in T_{\phi(x)}U \cong \mathfrak g^*_a$ annihilates $\mathfrak f$. 
 Now we look at $(2)$. 
For any $\xi, \eta\in\mathfrak g$, we have
$$\omega_x(\xi_{M, x}, \eta_{M, x})=\langle \xi, d\phi_x(\eta_{M, x})\rangle=\langle\xi, ad^*(\eta)\cdot\phi(x)\rangle = -\langle[\xi, \eta], \phi(x)\rangle.$$
This implies that $\mathfrak f_{M, x}$ is symplectic in $T_x M$
if and only if $ad^*(\mathfrak f)(\phi(x))$ is symplectic in $T_{\phi(x)}(G\cdot \phi(x))$. We have the splitting of the symplectic space $T_{\phi(x)}(G\cdot \phi(x))$:
$$T_{\phi(x)}(G\cdot \phi(x))=T_{\phi(x)}(G_a\cdot \phi(x))\oplus ad^*(\mathfrak f)(\phi(x)),$$ 
and $T_{\phi(x)}(G_a\cdot \phi(x))$ is a symplectic subspace. Moreover, $T_{\phi(x)}(G_a\cdot\phi(x))$ and $ad^*(\mathfrak f)(\phi(x))$ are symplectically perpendicular:  let $\xi\in \mathfrak g_a$, $\eta\in \mathfrak f$, then
$$\langle[\xi, \eta], \phi(x)\rangle = \langle \eta, ad^*(\xi)\phi(x)\rangle =0$$
since $\mathfrak f$ annihilates $T_{\phi(x)}U$.
Hence $ad^*(\mathfrak f)(\phi(x))$ is symplectic in $T_{\phi(x)}(G\cdot \phi(x))$. 
With (1) and (2), and by the fact $T_x M = T_x R\oplus\mathfrak f_{M, x}$, we can see that $\ker (\omega_x|_{T_x R}) = \ker (\omega_x)$, so the claim on the cross section follows. 

If the $G$-action is clean on $M$, then for any $x\in R$, 
$T_x(G\cdot x)\cap T_x\mathcal F \cong T_x(N\cdot x)$, where $N$ is the null group. Since $T_x(N\cdot x)\subset T_x\mathcal F \subset T_x R$ and  
$G_a$ is the maximal subgroup acting on $R$, we have
$T_x(G_a\cdot x)\cap T_x\mathcal F \cong T_x(N\cdot x)$.
\end{proof}

  The highest dimensional face $\tau^P$ of $\mathfrak t^*_+$ which intersects  $\phi(M)$ is called the {\bf principal face}. If $U^P$ is the slice at $\tau^P$, then the cross section
 $R^P=\phi^{-1}(U^P)=\phi^{-1}(\tau^P)$ is called the {\bf principal cross section}, on which only the maximal torus of $G$ acts.

\medskip

In the rest part of this section, we establish a convergence theorem for presymplectic clean Hamiltonian $G$-actions with proper moment maps.   

\begin{theorem}\label{retract}
Let $(M, \omega)$ be a connected presymplectic  Hamiltonian $G$-manifold with a clean $G$-action and proper moment map $\phi$. Assume $0\in \phi (M)$. Then a small open $G$-invariant neighborhood of 
$\phi^{-1}(0)$  equivariantly deformation retracts to $\phi^{-1}(0)$.
\end{theorem}

\begin{proof}
The case when $\ker (\omega) = 0$, i.e., the symplectic Hamiltonian $G$-action case, is proved  in \cite{Le, W}.   Here for the presymplectic case, we follow the same idea of proof in \cite{Le}. The main point is as follows: the action is clean equips us with the local normal form theorem (Theorem~\ref{nform}), which gives us local real analytic coordinates and allows us to write $\|\phi\|^2$ locally as a real analytic function,
then we can use the properties of local real analytic functions and the same arguments in \cite{Le} to prove the claim. 
 \end{proof}

\begin{corollary}\label{corretract}
Let $(M, \omega)$ be a connected presymplectic  Hamiltonian $G$-manifold with a clean $G$-action and proper moment map $\phi$. Assume $a\in \phi (M)$. Then a small open $G$-invariant neighborhood of $\phi^{-1}(G\cdot a)$  equivariantly deformation retracts to $\phi^{-1}(G\cdot a)$.
\end{corollary}

\begin{proof}
Let $U$ be the slice at $a$, and let $R=\phi^{-1}(U)$ be the $G_a$-invariant cross section. Since $a$ is a fixed point
of $G_a$, by  a shift of $\phi|_R$, we may think of $a$ as the value $0$ of $\phi|_R$ on $R$. By Theorem~\ref{cross}, the $G$-action on $M$ is clean implies
that the $G_a$-action on $R$ is clean, i.e., $T_x(G_a\cdot x)\cap T_x\mathcal F=T_x(N\cdot x)$ for all $x\in R$. By Theorem~\ref{retract}, a small open $G_a$-invariant neighborhood of $\phi^{-1}(a)$ in $R$ $G_a$-equivariantly deformation retracts to  $\phi^{-1}(a)$. By the equivariance of $\phi$, the claim follows.
\end{proof}

\section{The fundamental groups of the quotients and the proof of Theorem~\ref{C}}

In this section, we prove Theorem~\ref{C} on the fundamental groups of the quotients $M/G$ and the presymplectic quotients $M_a$s'. The method is similar to that of the proof of the theorem when $\ker(\omega) = 0$, i.e., the symplectic case (\cite{L}). For the current presymplectic case, we need to take a deeper look at the structure of connected compact Lie groups and their Lie algebras.
We use mainly two operations: removing strata from stratified spaces, and doing local deformation retractions using 
Theorem~\ref{retract} and Corollary~\ref{corretract}. For the removing process to work, we need to prove that the links of the removed strata are connected and simply connected.  For the local deformation retractions, we need the local properness of the moment map.

\medskip

A  {\bf stratified space} $X$ is a Hausdorff  and paracompact topological space defined recursively as follows: $X$ can be decomposed into a disjoint union of (locally finite) connected pieces, called strata, which are
 manifolds, such that given any point $x$ in a (connected) stratum $S$,  there exist an open neighborhood $U$ of $x$, an open ball $B$ around $x$ in $S$, a compact stratified space $L$, called the {\bf link of x}, and a homeomorphism 
$$B\times \overset{\circ}{C}L\longrightarrow U$$
 that preserves the decompositions. Here, $\overset{\circ}{C}L$ is 
a cone over the link $L$, i.e., 
$\big(L\times [0, \infty)\big)/L\times \{0\}$. We also call the  link of $x$ the {\bf link of S}.

We will use the following two useful results. Lemma~\ref{remove} can be checked using the Van-Kampen Theorem.
\begin{lemma}\label{remove}\cite{L0}
Let $X$ be a connected stratified space. If $X_0$ is a closed stratum in $X$ with connected and simply connected link,
then $\pi_1(X)\cong\pi_1(X-X_0)$.
\end{lemma}

 \begin{theorem}\label{arm}\cite{Ar}
   Let $K$ be a compact Lie group acting on a compact path connected and simply connected
   metric space $X$. Let $H$ be the smallest normal subgroup of $K$ which contains the identity component
   of $K$ and all those elements of $K$ which have fixed points.
   Then $\pi_1(X/K)\cong K/H$.
   \end{theorem}

Let $(M, \omega)$ be a connected presymplectic Hamiltonian $G$-manifold with a clean $G$-action and proper moment map $\phi$. 
Recall that by Theorem~\ref{conv}, $\phi(M)\cap \mathfrak t^*_+$ is a closed convex polyhedral set. We call a value $a$ of $\phi$ {\bf generic} if $\phi^{-1}(G\cdot a)$ consists of points with the smallest dimensional stabilizer groups on $M$. A connected set of generic values on the principal face of $\mathfrak t^*_+$ is called a {\bf chamber} of 
$\phi(M)\cap \mathfrak t^*_+$. 

\subsection{When $G=T$ is a torus}

\

\medskip

In this part, we prove Theorem~\ref{C} for $G=T$, a torus. We state the theorem in the abelian case as follows.

\begin{theorem 4(a)}
Let $(M, \omega)$ be a connected presymplectic Hamiltonian $T$-manifold with a clean $T$-action and proper moment map $\phi$. Then
we have natural isomorphisms $\pi_1(M/T)\cong\pi_1(M_a)$ for all $a\in \phi(M)$.
\end{theorem 4(a)}

In this case, let us keep in mind that $\phi(M)\cap \mathfrak t^*_+ = \phi(M)$, and it  is a locally finite closed convex polyhedral set.
Let $\mathcal H$ be the set of closed half spaces involved in $\phi(M)$.
 Each closed half space in $\mathcal H$ has an interior and a boundary,
we call them {\bf faces} of the closed half space. We call 
the intersections of the faces of the half spaces in $\mathcal H$ {\bf faces} of the polyhedral set $\phi(M)$. A face can consist of certain interior points of
the polyhedral set $\phi(M)$, or it can be lower dimensional internal or boundary set of intersection points of $\phi(M)$. The different dimensional faces of 
$\phi(M)$ are caused by different dimensional stabilizer groups of the action.  Lower dimensional faces other than the chambers defined above are called {\bf non-generic faces}.

\begin{lemma}\label{a=b}
Let $(M, \omega)$ be a connected presymplectic  Hamiltonian $T$-manifold with a clean $T$-action and proper moment map $\phi$. 
Then for any values $a$ and $b$ in the same chamber of $\phi(M)$,
we have $\pi_1(M_a)=\pi_1(M_b)$.
\end{lemma}

\begin{proof}
Assume $a$ and $b$ are in the same chamber $U$. Then $\phi\colon \phi^{-1}(U)\to U$ is a proper (equivariant) submersion. By Ehresmann's lemma, $\phi^{-1}(a)$ and $\phi^{-1}(b)$ are (equivariantly) diffeomorphic. So the claim follows.
\end{proof}

\begin{proposition}\label{Tlink}
Let $(M, \omega)$ be a connected presymplectic  Hamiltonian $T$-manifold with  a clean $T$-action and proper moment map $\phi$. 
Let  {\bf F} be a non-generic face of $\phi(M)$. 
Suppose $M_H$ is an orbit type such that $M_H\cap\phi^{-1}(\bf F)\neq \emptyset$. Let $\overline U$ be the closure of
one chamber $U$ such that ${\bf F}\subset \overline U$. Then the link 
$L_H$ of $\big(M_H\cap\phi^{-1}({\bf F})\big)/T$ in 
$\phi^{-1}(\overline U)/T$ is always connected and simply connected.
\end{proposition}

\begin{proof}
By Theorem~\ref{nform}, a neighborhood of an orbit with stabilizer group $H$ is isomorphic as a presymplectic Hamiltonian $T$-manifold to 
$$A=T\times_H(\mathfrak q^*\times S\times V),$$
where $\mathfrak q^*$, $S$ and $V$ are as stated in the theorem. 
Split $\mathfrak q^* \cong \R^l\times \R^m$, where $\R^m$ is the subspace which is mapped to $\bf F$, split $S=S^H\times S'$ and $V=V^H\times V'$, where $S^H$ and $V^H$ are the subspaces fixed by $H$. Either of these subspaces can be $0$ or the whole space. Then 
$$A_H\cap \phi^{-1}({\bf F}) = T\times_H(\R^m\times S^H\times V^H).$$ So
$$\big(A_H\cap \phi^{-1}({\bf F})\big)/T =\big(\R^m\times S^H\times V^H\big)/H  =\R^m\times S^H\times V^H.$$ 
We have
$$A\cap\phi^{-1}(\overline U) = T\times_H\Big(\R^m\times S^H\times V^H\times (\R^+)^l\times (S'\cap\psi^{-1}(\overline U))\times V'\Big),$$
where $\R^+$ denotes the nonnegative half space of $\R$.
Then 
$$\big(A\cap\phi^{-1}(\overline U)\big)/T =\big(\R^m\times S^H\times V^H\big)\times\big((\R^+)^l\times (S'\cap\psi^{-1}(\overline U))\times V'                                  \big)/H.$$ 
So the link of $\big(A_H\cap \phi^{-1}({\bf F})\big)/T$ in $\big(A\cap\phi^{-1}(\overline U)\big)/T$ is 
$$L_H=\mathcal S\Big((\R^+)^l\times (S'\cap\psi^{-1}(\overline U))\times V'\Big)/H,$$
where $\mathcal S(\cdot)$ denotes the sphere of the corresponding space.
This is the same as the link of $\big(M_H\cap \phi^{-1}({\bf F})\big)/T$ in  
$\phi^{-1}(\overline U)/T$. Since by Theorem~\ref{nform}, $V'$ has no contribution to $\im (\phi)$, and $\bf F$ is not a generic face, we have 
$$(\R^+)^l\times (S'\cap\psi^{-1}(\overline U))\neq 0.$$ 
 
 {\bf (1)} First assume $S'\cap\psi^{-1}(\overline U)=S'$. (This happens when a neighborhood of $\bf F$ meets only one chamber or when $\psi$ is trivial on $S'$.) Then
$$L_H=\mathcal S\big((\R^+)^l\times S'\times V'\big)/H.$$
If $(\R^+)^l\neq 0$, then $\mathcal S\big((\R^+)^l\times S'\times V'\big)$ is always connected and simply connected no matter what the vector spaces $S'$ and $V'$ are, and $H$ fixes $(\R^+)^l$, by Theorem~\ref{arm}, $L_H$ is connected and simply connected.
Next assume $(\R^+)^l = 0$, then $S'$ needs to be a nontrivial symplectic $H$-representation with a nontrivial moment map (in order to have the chamber).  We must have $\dim H > 0$. If $S'\cong\C$, and $V'=0$, then $L_H$ must be a point, hence is connected and simply connected. If $\dim (S'\times V') \geq 3$, then $\mathcal S(S'\times V')$ is connected and simply connected. 
By Theorem~\ref{arm}, $\mathcal S(S'\times V')/H^0$ is connected and simply connected, where $H^0$ is the identity component of $H$. 
Since each element of $\Gamma = H/H^0$ acts on $S'$ as an element of a circle, it must have a nonzero fixed point  in $\mathcal S(S')/H^0$, by Theorem~\ref{arm} again, $\mathcal S(S'\times V')/H^0/\Gamma =L_H$ is connected and simply connected.

 {\bf (2)}  Next assume that $S'\cap\psi^{-1}(\overline U)\subsetneq S'$. 

 {\bf (2a)}  Suppose first that $(\R^+)^l = 0$. Then $S'\cong \C^k$, with $k\geq 1$, is a nontrivial symplectic $H$-representation with nontrivial moment map $\psi$. With no loss of generality, we assume the $H$-action on $S'$ is effective, so $\dim (H)\leq k$. Since the $H$ action is linear and the moment map $\psi$ is homogeneous, to prove $\mathcal S\Big(\big(S'\cap\psi^{-1}(\overline U)\big)\times V'\Big)/H$ is connected and simply connected, we only need to prove 
$\Big(\big(\C^k\times V' - \{0\}\big)\cap\psi^{-1}(\overline U)\Big)/H$ is connected and simply connected. The moment map
for the closed subtorus $H$ action on the symplectic vector space $S'$    
$$\psi\colon S'\cong \C^k\longrightarrow \mathfrak h^*$$
is of the form $\psi (z_1, \cdots, z_k) = \sum_{i=1}^k |z_i|^2\alpha_i$, 
where the $\alpha_is$' are weight vectors in $\mathfrak h^*$. 
With no loss of generality, we may assume that the cone $\im (\psi)\cap \overline U$ is spanned by the first certain number of $\alpha_is$', and denote the index set of these $is$' by $J$, where $|J|\geq \dim (H)$. By writing the rest of the $\alpha_is$' as linear combinitions of the first linearly independant $\dim (H)$ number of
$\alpha_is$', we have that the map
$$\psi\colon S'\cong \C^k\longrightarrow \overline U$$
is given by $\psi (z_1, \cdots, z_k) = \sum_{i\in J} f_i\alpha_i$, where 
$f_i$ is obtained in the way above, and is of the form $$f_i (z) = \sum_{j=1}^k a_{ij} |z_j|^2 \geq 0.$$
Let  $A_i = \big\{(z, x)\in \C^k\times V'\,|\, f_i(z) > 0, f_j (z)\geq 0\,\, \mbox{for $j\neq i$}\big\}$. Then $$\Big(\big(\C^k\times V' - \{0\}\big)\cap\psi^{-1}(\overline U)\Big)/H = \bigcup_{i\in J} A_i/H.$$ 
We may argue that
each $A_i/H$ is connected and simply connected (as in $(1)$, we may argue that  $A_i/H^0$ is connected and simply connected, and then argue
that $A_i/H^0/\Gamma =A_i/H$ is connected and simply connected),
and $(A_i/H)\cap (A_j/H)$ is connected when $i\neq j$. Then by the Van-Kampen theorem, the above union set, hence $L_H$ is connected and simply connected. We leave this as an exercise, or we refer to the proof of \cite[Lemma 3.9]{L0}.

 {\bf (2b)}  Suppose next that $(\R^+)^l \neq 0$. The moment map
$$\phi\colon (\R^+)^l\times S' \times V'\longrightarrow  \overline U$$
is given by $\phi (r, z, x) = r + \psi (z)$, where $r = (r_1, \cdots, r_l)\in (\R^+)^l$, and the moment map $\psi$ on $S'$ is of the form $\psi = \sum_{i\in J} f_i\alpha_i$ similar to that in {\bf (2a)}. For $i\in J$, let 
$A'_i = \big\{(r, z, x)\in (\R^+)^l\times S' \times V'\,|\, r_j\geq 0                                                                                                                                                                                                                        \,\,\mbox{for all $j=1, \cdots, l$}, f_i > 0  \,\,\mbox{and}\,\, f_j\geq 0\,\,\mbox{for all $j\neq i$}\big\}$.
For $i = 1, \cdots, l$, let $B_i = \big\{(r, z, x)\in (\R^+)^l\times S' \times V'\,|\, r_i > 0, r_j\geq 0                                                                                                                                                                                                                      \,\,\mbox{for $j\neq i$}, \,\,\mbox{and}\,\, f_j\geq 0 \,\,\mbox{for all $j\in J$}\big\}$.   Then 
$$\qquad\left(\big((\R^+)^l\times S' \times V' - \{0\}\big)\cap\phi^{-1}(\overline U)\right)/H = \Big(\bigcup_{i\in J} A'_i/H\Big) \bigcup\Big(\bigcup_{i=1}^l B_i/H\Big).$$
Similar to the last case, we can show that this set, hence the link $L_H$ is connected and simply connected.
\end{proof}

The proofs of the following Lemma~\ref{c=a} and Proposition~\ref{TM=a} are the same as in \cite{L0}. For clarity and display of the ideas, we include the proofs here.

\begin{lemma}\label{c=a}
Let $(M, \omega)$ be a connected presymplectic  Hamiltonian $T$-manifold with  a clean $T$-action and proper moment map $\phi$. 
Let $c$ be a non-generic value, and $a$ be a generic value very near $c$.
Then $\pi_1(M_c)=\pi_1(M_a)$.
\end{lemma}

\begin{proof}
Let $O$ be a small open neighborhood of $c$ containing $a$, and $U$ be the chamber containing $a$.
Let $V=O\cap U$, and let $\overline V$ be the closure of $V$ in $O$. 
By Theorem~\ref{retract}, $\phi^{-1}(O)$, hence $\phi^{-1}(\overline V)$-equivariantly deformation retracts to  $\phi^{-1}(c)$, hence
$$\pi_1\big(\phi^{-1}(\overline V)/T\big)\cong\pi_1(M_c).$$
Let $B$ be the set of values in $\overline V - V$.
Using Lemma~\ref{remove}  and Proposition~\ref{Tlink} repeatedly in the right order, we get
$$\pi_1\big(\phi^{-1}(\overline V)/T\big)\cong
\pi_1\big(\phi^{-1}(\overline V)/T-\phi^{-1}(B)/T\big).$$
Since $\phi^{-1}(\overline V)/T-\phi^{-1}(B)/T$ deformation retracts to 
$\phi^{-1}(a)/T=M_a$, the claim follows.
\end{proof}

\begin{proposition}\label{TM=a}
Let $(M, \omega)$ be a connected presymplectic  Hamiltonian $T$-manifold with  a clean $T$-action and proper moment map $\phi$. 
Then $\pi_1(M/T)=\pi_1(M_a)$ for some value $a$.
\end{proposition}

\begin{proof}
Using deforming (Theorem~\ref{retract} and Corollary~\ref{corretract}) and removing (Lemma~\ref{remove} and Proposition~\ref{Tlink}) alternately, we can achieve the proof.  There can be different processes and different choices of the values $a$s'. Follow the same arguments as in the proof of \cite[Theorem 1.6]{L0} for $G=T$.
\end{proof}

Theorem 4(a) follows from Lemmas~\ref{a=b}, \ref{c=a}, and Proposition~\ref{TM=a}.

\subsection{When $G$ is nonabelian}
\
\medskip

In this part, we prove Theorem~\ref{C} for the case when $G$ is nonabelian. We state the theorem in the nonabelian case as follows.

\begin{theorem 4(b)}
Let $(M, \omega)$ be a connected presymplectic Hamiltonian $G$-manifold with a clean $G$-action and proper moment map $\phi$, where $G$ is nonabelian. Then
we have natural isomorphisms $\pi_1(M/G)\cong\pi_1(M_a)$ for all $a\in \phi(M)$.
\end{theorem 4(b)}

In this part, without specification, $G$ always denotes a connected compact {\it nonabelian} Lie group.

\smallskip

We first prove the following facts about Lie groups which will be important to us in the sequel.

\begin{proposition}\label{q}
Let $G$ be a connected compact semisimple nonabelian Lie group.
Let $H\subset G$ be a closed subgroup with Lie algebra $\mathfrak h$,
let $\mathfrak f =\mathfrak g/\mathfrak h$, and we may view $\mathfrak f$ as a direct summand of $\mathfrak g$ complementary to $\mathfrak h$. Let $\mathfrak a$ be an ideal of $\mathfrak g$, let 
$\mathfrak p = \mathfrak a/\mathfrak a\cap\mathfrak h$, and let 
$\mathfrak q = \mathfrak f/\mathfrak p$. 
\begin{enumerate}
\item If $\mathfrak q = 0$, then either $\mathfrak a =\mathfrak g$ or 
$H\subseteq G$ is a nonabelian subgroup. 
\item If $\mathfrak q\neq 0$, then $\dim (\mathfrak q)\geq 2$, and for the
adjoint action of $H$ on $\mathfrak q$, the smallest normal subgroup of $H$ containing the identity component of $H$ and those elements which 
have nonzero fixed points is $H$ itself. If 
$\dim (\mathfrak q) = 2$, then $\mathcal S(\mathfrak q)/H$ is a point, where $\mathcal S(\mathfrak q)$ denotes the sphere in $\mathfrak q$.
\end{enumerate}
\end{proposition}

\begin{proof}
Since $G$ is semisimple, we can split
$$\mathfrak g = \bigoplus_{i=1}^k \mathfrak s_i,$$ 
where each $\mathfrak s_i$ is a simple ideal
with $[\mathfrak s_i, \mathfrak s_j] = 0$ for $i\neq j$, and span$[\mathfrak s_i, \mathfrak s_i] =\mathfrak s_i$ (see \cite[Theorem 5.18]{S}).
Since $\mathfrak a$ is an ideal, it is a direct sum of some factors of 
$\mathfrak g$, with no loss of generality, we assume 
$$\mathfrak a = \bigoplus_{i=1}^l \mathfrak s_i \quad\mbox{with $l\leq k$}.$$ Since $H\subset G$ is a closed subgroup, 
$$\mathfrak h =\bigoplus_{i=1}^k \mathfrak h_i,$$ 
where $\mathfrak h_i\subseteq \mathfrak s_i$ is a subalgebra for each $i$.
Then $\mathfrak a\cap\mathfrak h = \bigoplus_{i=1}^l \mathfrak h_i$, so 
$$\mathfrak p = \mathfrak a/\mathfrak a\cap\mathfrak h = \bigoplus_{i=1}^l (\mathfrak s_i/\mathfrak h_i).$$ 
While
$$\mathfrak f = \mathfrak g/\mathfrak h = \bigoplus_{i=1}^k (\mathfrak s_i/\mathfrak h_i),$$  
so 
$$\mathfrak q =\mathfrak f /\mathfrak p = \bigoplus_{i=l+1}^k (\mathfrak s_i/\mathfrak h_i).$$ 

  {\bf  (1)}  Assume that $\mathfrak q =0$. Then either $\mathfrak f =0$ which means that $H=G$ hence $H$ is nonabelian, or $l=k$ which means that
$\mathfrak a =\mathfrak g$, or $l < k$ and $\mathfrak s_i =\mathfrak h_i$
for all $l+1\leq i\leq k$, which implies that $H$ is nonabelian (since $\mathfrak s_i$ is nonabelian).

 {\bf (2)}  Assume that $\mathfrak q \neq 0$. Then $l < k$, and there is at least one $i$ with $l+1\leq i\leq k$ so that $\mathfrak h_i\subsetneq \mathfrak s_i$. For each such $i$ with $\mathfrak h_i\subsetneq \mathfrak s_i$, since 
$\mathfrak s_i$ is nonabelian and $\mathfrak h_i$ is a subalgebra, $\dim(\mathfrak s_i/\mathfrak h_i)\geq 2$, so 
$\dim(\mathfrak q)\geq 2$.

Let $S_i = \exp\mathfrak s_i$. Then $G \cong (S_1\times\cdots\times S_k)/F$, where $F$ is a finite central subgroup of $G$ (\cite[Theorem 5.22]{S}).
So up to dividing a finite central subgroup, $H=H_1\times\cdots\times H_k$
with $H_i\subset S_i$ a subgroup for $1\leq i\leq k$. Let $H^0=H_1^0\times\cdots\times H_k^0$ be the identity component of $H$. There are finitely
many elements $gs$' of the form $g=(g_1, \cdots, g_k)\in H/H^0$, where for each $j$, $g_j$ is either $1$ or $g_j\notin H_j^0$.  Now consider a fixed $i$ above with  $l+1\leq i\leq k$ and $\mathfrak h_i\subsetneq \mathfrak s_i$. If $g_i=1$, then $Ad(g_i)X = X$ for all $X\in \mathfrak s_i/\mathfrak h_i$, hence $Ad(g)X=Ad(g_1, \cdots, g_k) X =Ad(g_1)\cdots Ad(g_k)X = X$ since $Ad(g_j)X=X$ for $j\neq i$ (due to the fact $[\mathfrak s_i, \mathfrak s_j]=0$ when $i\neq j$). If
$g_i\notin H_i^0$, then $H_i$ is nontrivial, and $g_i$ is in a maximal torus $T_i$ of $S_i$ but not
in a maximal torus of $H_i^0$, and $Ad(g_i)Y= Y$ for all 
$0\neq Y\in \mathfrak t_i=$ Lie$(T_i)$. So $Ad(g_i)Y' = Y'$, where $Y'\neq 0$ is the component of $Y$ in $\mathfrak s_i/\mathfrak h_i$, then similar to
the above,  $Ad(g) Y' = Y'$. We have shown that any $g\in H/H^0$ has a nonzero fixed point in $\mathfrak q$. So the smallest normal subgroup of $H$ containing the identity component of $H$ and all those elements which have nonzero fixed points is $H$ itself.

Now assume $\dim(\mathfrak q) = 2$. Then there is exactly one $i$ with 
$l+1\leq i\leq k$ such that $\mathfrak h_i\subsetneq \mathfrak s_i$, and
$\mathfrak q = \mathfrak s_i/\mathfrak h_i$, with $\mathfrak h_i\neq 0$. If we consider the real root space decomposition of respectively $\mathfrak h_i$ and $\mathfrak s_i$, 
we can see that the Cartan subalgebra of $\mathfrak h_i$ and of 
$\mathfrak s_i$ must be the same, and the space $\mathfrak q$ can be identified with a $2$-dimensional (nonzero) root space of $\mathfrak s_i$.
Let $X$ and $Y$ be  two suitable vectors in this $2$-dimensional root space. Then there is a nonzero element $Z$ in the Cartan subalgebra
of $\mathfrak s_i$ so that $X$, $Y$, and $Z$ generate a Lie algebra isomorphic to that of $SU(2)$ (see p127 of \cite{S}).  The one parameter subgroup of $H$ generated by $Z$ acts on $\mathcal S(\mathfrak q)$ transitively, hence  $\mathcal S(\mathfrak q)/H$ is a point.
\end{proof}

Now we proceed with the steps of the proof of Theorem 4(b).

\begin{proposition}\label{Glink}
Let $(M, \omega)$ be a connected presymplectic Hamiltonian $G$-manifold with a clean $G$-action and proper moment map $\phi$.   Let 
$\mathcal C$ be the (closed) central face of $\mathfrak t^*_+$, and assume that $\mathcal C\cap\phi (M)\neq \emptyset$ and
that $\mathcal C$ is not the only face of $\mathfrak t^*_+$ which intersects $\phi(M)$. For each orbit type $M_{(H)}$ such that $M_{(H)}\cap\phi^{-1}(\mathcal C)\neq\emptyset$, let $L_H$ be the
link of $\left(M_{(H)}\cap\phi^{-1}(\mathcal C)\right)/G$ in $M/G$.
 Then $L_H$ is connected and simply connected.
\end{proposition}

\begin{proof}
We write $G=(G_1\times T_c)/F$,
where $G_1$ is a connected compact semisimple Lie group, $T_c$ is a connected (central) torus, and $F$ is a finite central subgroup (\cite[Theorem 5.22]{S}). Let $\mathfrak g_1$ and $\mathfrak t_c$ be respectively the Lie algebras of $G_1$ and $T_c$,
and $\mathfrak g_1^*$ and $\mathfrak t_c^*$ be their dual Lie algebras.

 By Theorem~\ref{nform}, a neighborhood of a $G$-orbit in $\phi^{-1}(\mathcal C)$ with stabilizer group $(H)$ is isomorphic to 
$$A=G\times_H(\mathfrak q^*\times S\times V),$$ 
where $\mathfrak q^*$, $S$ and $V$ are as explained in Theorem~\ref{nform} (and the paragraph above it). Up to dividing a finite {\it central} subgroup, $H=H_1\times T_1$, where $H_1\subset G_1$ and $T_1\subset T_c$ are closed subgroups (since $H$ is closed). Let $\mathfrak h$, $\mathfrak h_1$ and $\mathfrak t_1$ be respectively the Lie algebras of $H$, $H_1$ and $T_1$. Under the splitting of $G$ above, let the null ideal 
$$\mathfrak n = \mathfrak n_1\oplus \mathfrak n_2, \,\,\,\mbox{where $\mathfrak n_1\subseteq\mathfrak g_1$ and $\mathfrak n_2
\subseteq\mathfrak t_c$ are ideals}.$$
Then 
$$\mathfrak p = \mathfrak n/\mathfrak n\cap\mathfrak h = \big(\mathfrak n_1/(\mathfrak n_1\cap\mathfrak h_1)\big)\oplus \big(\mathfrak n_2/(\mathfrak n_2\cap\mathfrak t_1)\big) \triangleq\colon \mathfrak p_1\oplus\mathfrak p_2.$$ The stabilizer group under the coadjoint action of each point on the central face $\mathcal C$ is $G$, so 
$$\mathfrak m = \mathfrak g/\mathfrak h = (\mathfrak g_1/\mathfrak h_1)\oplus (\mathfrak t_c/\mathfrak t_1)\triangleq\colon\mathfrak m_1\oplus\mathfrak m_2.$$ 
Then 
$$\mathfrak q=\mathfrak m/\mathfrak p = (\mathfrak m_1/\mathfrak p_1)\oplus (\mathfrak m_2/\mathfrak p_2) \triangleq\colon \mathfrak q_1\oplus\mathfrak q_2.$$
So we can write
$$A = G\times_H\big((\mathfrak q_1^*\times\mathfrak q_2^*)\times S\times V\big).$$
By the moment map description on $A$, we have 
$$A_{(H)}\cap\phi^{-1}(\mathcal C) = G\times_H\big(\mathfrak q_2^*\times S^H\times V^H\big),$$ 
where $S^H$ and $V^H$ are respectively the subspaces
of $S$ and $V$ fixed by $H$. So 
$$\big(A_{(H)}\cap\phi^{-1}(\mathcal C)\big)/G =\big(\mathfrak q_2^*\times S^H\times V^H\big)/H =\mathfrak q_2^*\times S^H\times V^H.$$ 
While $$A/G =\big(\mathfrak q_2^*\times S^H\times V^H\big)\times\big((\mathfrak q_1^*\times S'\times V')/H\big),$$ 
where $S'$ and $V'$
are respectively the complementary subspaces of $S^H$ in $S$ and $V^H$   
in $V$. The link of $\big(A_{(H)}\cap\phi^{-1}(\mathcal C)\big)/G$ in 
$A/G$ is
$$L_H = \mathcal S(\mathfrak q_1^*\times S'\times V')/H,$$
where $\mathcal S(\cdot)$ denotes the sphere in the corresponding space.
This is the same as the link of $\big(M_{(H)}\cap\phi^{-1}(\mathcal C)\big)/G$ in $M/G$.
By assumption, $\phi(M)$ intersects at least another higher dimensional face of  $\mathfrak t^*_+$. Let $a\in\phi(M)$ be 
on this higher dimensional face, then $G_a\cap G_1\subsetneq G_1$ (otherwise,
$G_1\subseteq G_a$, then $G=G_a$, so $a$ is in $\mathcal C$, a contradiction). Since the coadjoint 
orbit  $G\cdot a$ lies on the affine space spanned by $\mathfrak n^{\circ}$ (by Proposition~\ref{aspan}), we have $\mathfrak n\subseteq\mathfrak g_a=\, $Lie$(G_a)$, hence $\mathfrak n_1\subsetneq \mathfrak g_1$. Then by Proposition~\ref{q} applied for the semisimple group $G_1$, we have $2$ possibilities: 
\begin{enumerate}
\item   $\mathfrak q_1^* = 0$ and $H$ is nonabelian, and 
\item  $\dim(\mathfrak q_1^*) \geq 2$ and we have the claims in part $(2)$ of Proposition~\ref{q}. 
\end{enumerate}
First assume we are in case $(1)$. Then $S'$
must be a nontrivial symplectic $H$-representation with a nontrivial moment map, hence is of dimension at least $4$.
So $\mathcal S(S'\times V')$ is connected and simply connected. By Theorem~\ref{arm}, $\mathcal S(S'\times V')/H^0$ is connected and simply connected, where
$H^0$ is the identity component of $H$. Since each element in 
$\Gamma=H/H^0$ acts on $S' \cong \C\times\cdots \times \C$ as
an element of a circle, each element in $\Gamma$   has a fixed point in $\mathcal S(S')/H^0$. By Theorem~\ref{arm} again, $\big(\mathcal S(S'\times V')/H^0\big)/\Gamma =\mathcal S(S'\times V')/H = L_H $ is connected and simply connected.
Now assume we are in case $(2)$. If 
$\dim (\mathfrak q_1^*\times S'\times V') >2$, then $\mathcal S(\mathfrak q_1^*\times S'\times V')$ is connected and simply connected. Note that
the central component $T_1$ of $H$ fixes $\mathfrak q_1^*$. By
Proposition~\ref{q} $(2)$ and Theorem~\ref{arm}, $L_H$ is connected and simply connected. If $\dim (\mathfrak q_1^*\times S'\times V') = 2$, i.e.,
$S'=V'=0$ and $\dim (\mathfrak q_1^*) = 2$, then by Proposition~\ref{q} 
$(2)$, $L_H$ is a point hence is connected and simply connected.
\end{proof}

The proofs of the following Lemmas~\ref{o-b}, \ref{Gisom} and \ref{Gisom'}  are the same as in \cite{L0} and \cite{L}. In the proof of Lemma~\ref{o-b}, we use the cross section theorem Theorem~\ref{cross}. Here we omit their proofs.

\begin{lemma}\label{o-b}
 Let $(M, \omega)$ be a connected presymplectic Hamiltonian $G$-manifold with a clean $G$-action and proper moment map $\phi$. Let $c\in\tau\cap\phi(M)$, where $\tau\neq\tau^P$ is a
face of $\mathfrak t^*_+$, $\tau^P$ being the principal face, and let $a$ be a generic value on $\tau^P$ very near $c$.
Let $O$ be a small open invariant neighborhood
of $c$ in $\mathfrak g^*$ containing $a$. Let $B$ be the set of values in $O\cap \mathfrak t^*_+$  other than those on the chamber of generic values containing $a$ on $\tau^P$.
Then
$$\pi_1\big(\phi^{-1}(O)/G\big)\cong\pi_1\big(\phi^{-1}(O)/G-\phi^{-1}(G\cdot
   B)/G\big).$$
\end{lemma}

\begin{lemma}\label{Gisom}
Let $(M, \omega)$ be a connected presymplectic Hamiltonian $G$-manifold with a clean $G$-action and proper moment map $\phi$.  Let 
$c\in\tau\cap\phi(M)$, where $\tau\neq \tau^P$ is a face of $\mathfrak t^*_+$, $\tau^P$ being the principal face,  and let $a$ be a generic value on $\tau^P$ very near $c$. Then $\pi_1(M_c)\cong\pi_1(M_a)$.
 \end{lemma}

\begin{lemma}\label{Gisom'}
Let $(M, \omega)$ be a connected presymplectic Hamiltonian $G$-manifold with a clean $G$-action and proper moment map $\phi$. Let  
$\tau^P\subset \mathfrak t^*_+\cap\phi(M)$ be the principal face. Then
$\pi_1\big(\phi^{-1}(G\cdot\tau^P)/G\big)\cong\pi_1(M_a)$ for all $a\in\tau^P$.
\end{lemma}

Now we can finish the proof of Theorem 4(b):

\begin{proof}
Similar to the proof of Lemma~\ref{o-b}, by going to the cross sections, 
using Proposition~\ref{Glink} in the cross sections, and by equivariance of the action and of the moment map, we can inductively remove $\phi^{-1}(G\cdot \tau)/G$'s from $M/G$ for all the nonprincipal faces $\tau $'s of $\mathfrak t^*_+$  (note that the links are local). Now assume we have achieved 
that 
$$\pi_1(M/G)\cong\pi_1\Big(M/G - \bigcup_{\tau\neq\tau^P}\phi^{-1}(G\cdot \tau)/G\Big)\cong\pi_1\big(\phi^{-1}(G\cdot\tau^P)/G\big).$$
Then the theorem follows from Lemmas~\ref{Gisom} and \ref{Gisom'}.
\end{proof}

\section{an application of Theorems~\ref{A} and \ref{C} and two counter examples}

First we look at an example of Theorems~\ref{A} and \ref{C}.

\begin{example}\label{ex0}
Let $(M, \omega)$ be a compact connected symplectic manifold of dimension $2n$ equipped with a Hamiltonian $G$-action and moment map $\phi$. Choose a $G$-invariant Riemannian metric on $M$ compatible with $\omega$. Let $Y$ be the unitary frame bundle of $M$,
$\pi\colon Y\to M$,  then $(Y, \pi^*(\omega))$ is a compact connected presymplectic manifold whose leaves are the $U(n)$-orbits. The $G$-action on $M$ naturally lifts to a $G$-action on $Y$, which commutes with the $U(n)$ action. The $G\times U(n)$-action is leafwise transitive, hence clean on $Y$, with null subgroup $N\supseteq U(n)$. The action of $G\times U(n)$ is Hamiltonian on $Y$ with moment map 
$$\Phi = \phi\times 0\colon \, Y\longrightarrow \mathfrak g^*\times \mathfrak u(n)^*.$$
 Let $\mathcal O$ be an $N$-orbit, which is diffeomorphic to $U(n)$. By Theorem~\ref{A},
\begin{equation}\label{y}
\pi_1(Y)/\langle\im(\pi_1(\mathcal O))\rangle =\pi_1(Y/(G\times U(n))) = \pi_1(M/G).
\end{equation}
By Theorem~\ref{C},
$$\pi_1(Y/(G\times U(n))) = \pi_1(Y_a) \quad\mbox{for any $a\in\Phi(Y)$.}$$
If $a\in \phi(M)$, then $a\in\Phi(Y)$, and $Y_a = \Phi^{-1}\big((G\times U(n))\cdot a\big)/\big(G\times U(n)\big) = \Phi^{-1}(G\cdot a)/\big(G\times U(n)\big) = \phi^{-1}(G\cdot a)/G = M_a$. So
\begin{equation}\label{ya}
\pi_1(Y/(G\times U(n))) = \pi_1(M/G) = \pi_1(M_a) \quad\mbox{for any $a\in\phi(M)$.}
\end{equation}
On the other hand, by the long exact sequence of the $U(n)$-fibration $\pi\colon Y\to M$,
$$\cdots\to\pi_1(U(n))\to\pi_1(Y)\to\pi_1(M)\to\pi_0(U(n))=0\to\cdots,$$
we get
\begin{equation}\label{u}
\pi_1(Y)/\im(\pi_1(U(n))) = \pi_1(M).
\end{equation}
Since the $N$-orbit $\mathcal O$ is diffeomorphic to $U(n)$, we have 
\begin{equation}\label{yu} 
\pi_1(Y)/\langle\im(\pi_1(\mathcal O))\rangle = \pi_1(Y)/\im(\pi_1(U(n))).
\end{equation}
Since $M$ is a compact connected Hamiltonian $G$-manifold, by \cite{L0},
$$\pi_1(M) = \pi_1(M/G).$$
This shows the consistance of  (\ref{y}) and (\ref{u}). 
Equations (\ref{y}), (\ref{ya}), (\ref{u}) and (\ref{yu}) recover the claim for the compact symplectic manifold $M$  (\cite{L0}):
$$\pi_1(M) = \pi_1(M/G) = \pi_1(M_a)  \quad\mbox{for any $a\in\phi(M)$}.$$
\end{example}

\smallskip

Now we look at two counter examples of the theorems.

\begin{example}\label{ex1}
Let $M=S^1\times T^2$, and $\alpha=\cos t \,d\theta_1 + \sin t \,d\theta_2$,
where $t$ is the coordinate on the first factor and $(\theta_1, \theta_2)$
are the coordinates on the second factor. Then $(M, \alpha)$ is a contact
manifold, and $(M, d\alpha)$ is presymplectic.
The Reeb vector field is $R=\cos t\frac{\partial}{\partial \theta_1} + \sin t\frac{\partial}{\partial \theta_2}$.
The null foliation on $M$ is given by the orbits of the flow of $R$.

Let $T^2$ act on $M$ by acting freely on the second factor and acting trivially on the first factor. This $T^2$-action is {\it not clean}, one can check it by definition. Let us look at the moment map image. The moment map for the $T^2$-action is $\phi (t, \theta_1, \theta_2)=(\cos t, \sin t)$, $\phi(M)$ is a circle, not a convex polytope. 
For any  $a\in \phi (M)$,
$M_a = (\mbox{pt}\times T^2)/T^2 = \mbox{pt}$, so $\pi_1(M_a)=0$.
But $M/T^2 = S^1$, so $\pi_1(M/T^2)=\Z$.
\end{example}

\begin{example}\label{ex2}
Consider the contact manifold in Example~\ref{ex1}.
Let $S^1\subset T^2$ act on $M$ by acting freely on the first coordinate
of $T^2$. This $S^1$-action is {\it not clean}, checking by definition. The moment map of this $S^1$-action is $\phi (t, \theta_1, \theta_2) = \cos t$, so $\phi (M) = [-1, 1]$. We have that
$$M/S^1 = S^1\times S^1,$$
$$\phi^{-1}(0)= \mbox{2 points}\times T^2, \,\, \mbox{so}\,\, M_0 = \mbox{2 points}\times S^1, \mbox{and}$$
$$\phi^{-1}(1)=  \mbox{1 point} \times T^2, \,\, \mbox{so}\,\,  M_1 = S^1.$$
\end{example}

\end{document}